\documentclass[11pt]{amsart}

\usepackage[margin=1in]{geometry}
\usepackage{graphicx,subcaption,mathtools,url,amsmath,amssymb}

\usepackage{amsmath}
\usepackage{amsthm}
\usepackage{hyperref}
\usepackage{graphicx}
\usepackage{subcaption}
\usepackage{bm}

\newtheorem{theorem}{Theorem}

\DeclareMathOperator{\NumPDESolve}{NumSolve}
\DeclareMathOperator{\NNOperator}{NNOperator}
\newcommand{\vx}{\bm{x}}
\begin{document}
\title{Numerical PDE solvers outperform neural PDE solvers}
\author{Patrick Chatain \and Michael Rizvi‑Martel \and Guillaume Rabusseau \and Adam Oberman}
\date{\today}
%\titlerunning{Comparison of neural and numerical PDE solvers} 

%\author{Patrick Chatain  \and
%        Michael Rizvi-Martel \and
%        Guillaume Rabusseau \and
%        Adam Oberman
%}
%
%\institute{Patrick Chatain \at
%              Dept.\ of Mathematics \& Statistics, McGill University, Montr\'eal, QC H3A~0B9, Canada\\
%              \email{patrick.chatain@mail.mcgill.ca}
%          \and
%          Michael Rizvi-Martel \at
%             Mila -- Quebec AI Institute, Montr\'eal, QC, Canada\\
%             \email{michael.rizvi-martel@mila.quebec}
%          \and
%          Guillaume Rabusseau \at
%             Mila -- Quebec AI Institute and DIRO, Universit\'e de Montr\'eal, Montréal, QC, Canada\\
%             \email{rabussgu@mila.quebec}
%          \and
%          Adam Oberman \at
%             Dept.\ of Mathematics \& Statistics, McGill University, Montréal, QC, Canada\\
%             \email{adam.oberman@mcgill.ca}
\maketitle

\begin{abstract}
We present DeepFDM, a differentiable finite‑difference framework for learning spatially varying coefficients in time‑dependent partial differential equations (PDEs). By embedding a classical forward‑Euler discretization into a convolutional architecture, DeepFDM enforces stability and first‑order convergence via CFL‑compliant coefficient parametrizations. Model weights correspond directly to PDE coefficients, yielding an interpretable inverse‑problem formulation. We evaluate DeepFDM on a benchmark suite of scalar PDEs: advection, diffusion, advection–diffusion, reaction–diffusion and inhomogeneous Burgers’ equations—in one, two and three spatial dimensions. In both in‑distribution and out‑of‑distribution tests (quantified by the Hellinger distance between coefficient priors), DeepFDM attains normalized mean‑squared errors one to two orders of magnitude smaller than Fourier Neural Operators, U‑Nets and ResNets; requires 10–20× fewer training epochs; and uses 5–50× fewer parameters. Moreover, recovered coefficient fields accurately match ground‑truth parameters. These results establish DeepFDM as a robust, efficient, and transparent baseline for data‑driven solution and identification of parametric PDEs.

%\keywords{neural operators \and finite difference methods \and inverse problems \and out‑of‑distribution generalization \and scientific machine learning}
%\subclass{35R30 \and 65M06 \and 65M32 \and 65C20 \and 68T07}
\end{abstract}

\section{Introduction}
Neural networks have recently been applied in various ways to solve Partial Differential Equations (PDEs), \cite{Karniadakis_2021,lu2021deepxde,Li:2020aa,Li:2020ab}.  Neural network methods offer a highly flexible and adaptive approach to solving PDEs, compared to traditional numerical PDE solvers, as they can work with a wide variety of equations and  input data. 
They are more flexible than standard numerical solvers, because they treat the problem of solving PDEs as a problem of learning from data. This model-free approach allows for PDEs to be solved without explicit knowledge of the coefficients (or even the terms) in the PDE.  Moreover, they allow users to compute approximate solutions without the technical skills needed to implement standard PDE solvers.  As a result, these methods have been very popular, with thousands of recent citations for the works cited above. 
However, the learned operators do not correspond to solution operators of the underlying PDE from which the data was generated. 

While the neural network approaches  have been extensively compared against each other, they have not, until now, been carefully compared to established numerical PDE methods.  In this article we set out to carefully compare these approaches.  

Reproducible code for this paper can be found at \url{https://github.com/adam-oberman/deep-pdes}
\subsection{Benchmark PDEs}\label{pdesIC}
%\paragraph{PDE setup}
It is necessary to specify a class of PDEs to be solved in order to compare the neural network solvers with standard numerical PDE methods. 
%We are restricted by the requirements of PDE solvers which require a consistent type of equation (e.g. time dependent versus stationary).   
The choice of PDEs was based on choosing a maximal subset of those considered in the paper \cite{Li:2020aa}.  Because our goal was to build a single model which could learn any of the included PDEs, we required that they all have the same data structure, in the sense of being scalar equations (rather than vector equations), and time-dependent (rather than stationary).  This required us to exclude the two dimensional Stokes equations, because it is a vector PDE, and the Darcy equation, because it is not time-dependent.  The rest of the PDEs considered in the paper are addressed here. 

In order to better explore the performance differences of the methods, we then extended these PDEs to a wider class, and further studied performance under (i) increased coefficient variability, and (ii) shift in the data-generation distribution (via different Fourier-coefficients priors).  We also implemented one and three dimensional solvers. 

To make the notation concise, we write $\vx$ for the vector spatial variable, which is two dimensional in most cases.  We considered the following PDEs.

Advection equation (non-diffusive), 
\[
\partial_t u(\vx,t)=b(\vx) \cdot \nabla u(\vx,t)
\]
Diffusion equation (inhomogeneous), 
\[\partial_t u(\vx,t)=a(\vx) \Delta u(\vx,t)\]
Advection-Diffusion, $$\partial_t u(\vx,t)=b(\vx) \cdot \nabla u(\vx,t)+a(\vx) \Delta u(\vx,t)$$
Reaction diffusion equation, $$\partial_t u(\vx,t)=a(\vx) \Delta u(\vx,t)+c(\vx) u(\vx,t)(1-u(\vx,t))$$
Burgers' equation (inhomogeneous), $$\partial_t u(\vx,t)=b(\vx) u(\vx,t) \cdot \nabla u(\vx,t)+a(\vx) \Delta u(\vx,t)$$
Here $\nabla$ is the spatial gradient operator and $\Delta$ is the spatial Laplacian. 

We implemented a model which was designed to learn finite difference solution operators from a broader class of PDEs.  When given data corresponding to solutions of one of the PDEs above, the solver did not know the form of the PDE.  Rather, we implemented a single model capable of learning a solution operator for a class of time-dependent PDEs which covers all the cases above.  It is described in the following section.

\subsection{Parametric PDE family}
We considered the nonlinear PDE,
\begin{equation}\label{PDE} \tag{PDE}
    \partial_t u(\vx, t)
    %P(u(\vx, t); a(\vx), b(\vx)) 
    = L\left(u(\vx, t) ; a(\vx) \right)+N\left(u(\vx, t) ; b(\vx)\right)
\end{equation}
where the PDE operator as a combination of linear operators and nonlinear operators parameterized by $a(\vx)$, and $b(\vx)$, which are vectors of PDE coefficients, 
\[
    P\left(u(\vx, t) ; a(\vx), b(\vx) \right)=L\left(u(\vx, t) ; a(\vx) \right)+N\left(u(\vx, t) ; b(\vx)\right)
\]
These coefficients are assumed to be bounded, with known upper and lower bounds.
The linear part of the operator is given by 
\begin{equation}\label{PDEcoeffs}
	\begin{aligned}
			L(u(\vx,t), a(\vx)) &=  a_0 (\vx) + a_1 (x)u(\vx,t)  + a_2 (\vx) \cdot \nabla  u(\vx,t)
   + a_3 (x) \Delta u(\vx,t)
	\end{aligned}
\end{equation}
%with coefficients $a_i(\vx)$, which are bounded above and below, by known values. 
%
The nonlinear part corresponds to two terms,
\[
N(u(\vx,t) ; b(\vx)) =   b_1(\vx) u(\vx,t)(1-(u(\vx,t))  +  b_2(\vx)u(\vx,t) \nabla u(\vx,t) 
\]
the first is the reaction term, and the second term corresponds to the Burgers' equations.  Whenever this term is nonzero, we ensure that there is also a diffusion term with coefficients bounded away from zero, which ensures the solution is smooth enough.   

 We require that, for all admissible coefficients, the PDE problem is well-posed, in the PDE sense, meaning that for admissible initial and boundary data, $u_0(x)$, there exist unique solutions which are stable in a specified norm. 

The operator above includes, by setting some coefficients to zero,  each of the PDEs discussed in the previous section. 

We focus mainly on the two dimensional case, but we also include a one dimensional and a three dimensional implementation.   We also present one dimensional results below, for illustration purposes.

\section{Related work}

\subsection{Neural PDEs}

Early machine learning methods \cite{rudy2017data} focused on data discovery of PDEs, without building solution operators.   Physics-informed neural networks (PINNs) \cite{Karniadakis_2021,shin2020convergence} were  among the first models to leverage neural networks for solving PDEs.  These methods learn a specific \textit{instance} of the PDE, rather than a solution operator.   %\cite{zubov2021neuralpde} is also a single instance solver, implemented with PINNs.

\paragraph{Key Neural PDE Operator learning papers}
\cite{lu2019deeponet} propose the DeepONet architecture, which learn PDE solution operators.  However, in this case, the PDE is fully known, and the PDE residual is included in the loss. Subsequently, \cite{Li:2020ab} proposed a similar approach by learning Green's function for a given PDE. This method gave rise to the Fourier neural operator \cite{Li:2020aa}, which leverages certain assumptions on Green's function to solve the problem in Fourier space. \cite{liu2022predicting} build neural network models which integrate PDE operators directly in the model's architecture, while retaining the large capacity neural network architecture.

There are too  many recent works on neural PDEs to mention. \cite{zhang_neuralpdesolver_2024} lists over seven hundred papers, most of them from the last three years, including survey and benchmark papers. We mention a few relevant papers:  \cite{DBLP:conf/nips/TakamotoPLMAPN22}, provide a benchmark dataset and interface for learning PDEs.  Solver in the loop \cite{DBLP:conf/nips/UmBFHT20} integrate NN methods with a PDE solver.  ClimODE \cite{DBLP:conf/iclr/VermaH024} solves an advection equation with source.

%\paragraph{Numerical PDE solver and inverse problems} Traditional Numerical PDE solvers:	\cite{guyer2009fipy} implement a PDE solver software package which covers advection, diffusion, and reaction equations.   The Level Set toolbox \cite{mitchell2008flexible} is a package which covers more nonlinear PDEs, including the eikonal equation. 

%  \cite{hwang2022solving} solve optimal control problems with PDE constraints,  
 
\paragraph{Neural-networks architectures related to differential equations}
Several works connect neural network architectures and solution operators for differential equations. \cite{chen2018neural} proposed a neural network architecture based on ODE solvers and \cite{Haber_2017} focused on the stability aspects of the architecture.  Also have Neural SDEs \cite{tzen2019neural}. 
\cite{Ruthotto2020} proposed network architectures based on discretized PDE solvers, however they do not learn PDE operators.  The early contribution \cite{long2018pde} represents a PDE solution operator as a feed-forward neural network, and learns both an approximation to the PDE coefficients as well as the solution operator from the data, however this method had very low accuracy. 
%. \cite{long2019pde} extends to also recovering the approximate symbolic form of the equation. However, these early contributions had very low accuracy.  

\subsection{Inverse Problems}
Neural PDE operators aim to learn to solve a given PDE from data, without assuming that the form of the PDE is known.  In contrast, the PDE inverse problem approach assumes that a specific form of the PDE in known, but that the coefficients are unknown.
  More precisely, PDE  inverse problems  \cite{taler2006solving} aim to infer unknown parameters of a PDE with a known form, using a dataset of PDE solutions. 
 While the inverse problem approach is compatible, in theory, with any numerical or neural PDE method, the drawback is the specialized nature of each solver,  with custom code and custom optimization routines.

\paragraph{Numerical inverse problems}
We first  discuss the approach which uses forward PDE solvers to learn the PDE coefficients. 
The PDE approach uses numerical PDE solvers, such as  Finite Element Methods and Finite Difference Methods~\cite{Stigg}. These methods discretize the PDEs over the domain into a system of equations that can be solved numerically. 
%Numerical PDE solvers are an established scientific computing tool, which provide stable, accurate PDE solution operators. However, 
These  solvers require knowledge of the full equation governing the process of interest, and operate on structured input data, which can limit their applicability and adaptability to a wider range of scientific problems.  For inverse problems, numerical PDE solvers are often used in combination with gradient-based optimization techniques.  These are implemented in packages such as \cite{comsol2023,logg2012automated,2020SciPy-NMeth,doi:10.1137/16M1081063}.  
These methods are computationally intensive, and often require customized code for each problem formulation. However, when combined with proper regularization and optimization strategies, they provide accurate and reliable solutions.

 \paragraph{Neural inverse problems} 
 There are a number of works in neural inverse problems.  (These works differ somewhat from ours, in that they focus on a single problem at a time, rather than developing a methodology for solving a wide class of inverse problems - they have not applied the same method to a number of benchmark problems). 
  \cite{DBLP:conf/icml/ZhaoLW22} solve PDE inverse problems, such as waveform inversion.  In this case, the forward solver is given by a graph neural network or by a U-Net.  They report faster solution times, compared to using the Finite Element Method for the forward solver.  Even using neural networks, their approach has the limitation that changes to the PDEs require training a new forward solver. \cite{huang2022efficient} treat inverse problems for Darcy and Navier Stokes.
\cite{jiao2024solving} use DeepONets as a solver in a Bayesian Markov Chain Monte Carlo (MCMC) approach to PDE inverse problems, to learn from noisy solutions of a diffusion equation. 
\cite{zhang2024bilo} solves inverse problems using a PINN approach for the forward solver.
  
A second approach is the Bayesian inference approach to inverse problems \cite{stuart2010inverse}: this one is more appropriate for problems with uncertainty in the model and noise in the data, it does not require PDE solvers. \cite{cao2023residual} solve Bayesian Inverse Problems, they find that using neural networks for the forward solver are faster but less acurate, compared to traditional scientific computing solvers, so they implement a hybrid approach.

\section{Neural Operators as learned PDE solvers}\label{sec:learningPDEsolvers}
In this and the following section, we explain the neural PDE operators, followed by an explanation of the learned numerical PDE solver approach. 

 A neural PDE operator such as FNO also takes as input a dataset of time slices of solutions of \eqref{PDE}.  It learns to fit the data, using an architecture which is compatible with PDE solutions (e.g. by enforcing a smoothness to the solutions).  However, it does not learn a specific PDE solution operator, and it may not preserve important properties of the solution operator (e.g. linear superposition, or the maximum principle, or conservation of mass, etc.).   Moreover, as we demonstrate below, the FNO solution does not generalize as well as the numerical PDE solution.  

Given the same data, the parametric finite difference solver, 
learns to fit the data to a solution operator for the time-dependent PDE, \autoref{PDE}, above, by performing linear regression over the coefficients of each term of the PDE family.

\subsection{Data generation for PDE solutions}

We are given a dataset consisting of sample values functions of the form $u(\vx, t)$ for $\vx$ on a grid, and for several equally spaced values (time slices) of $t$.  So let $G=[0,1]^2$ be the unit square and consider a uniform grid, $G^x$ in space of resolution determined by $N_{\vx}$. The number of output time slices are given by $N_T$, uniformly spaced time values from $[0, T]$. Then grid functions are written
$$
\begin{aligned}
U_{0, i}\left(\vx_j\right) & =u_{0, i}\left(\vx_j\right), & & x \in G^x \\
U_i\left(\vx_j, t_k\right) & =u_i\left(\vx_j, t_k\right), & & x \in G^x, t_k \in T^h
\end{aligned}
$$
Thus the dataset is, $S^m = \{ U_1, \dots, U_m \}$ where each element 
$$
U_i(X,T) =  (U_i(X,0), U_i(X,t_1), \dots, U_i(X,t_k))
$$ 
consists of a vector of grid values of a PDE, one for each time, $t \in T = (0,t_1,\dots, t_k=T)$.  
The dataset comes from two sources. First, when available, we use  benchmark data provided by code from neural PDE papers.  However,  
we also generate data by  solving the numerical PDE solver \eqref{NumPDE}, on a higher resolution grid, and then coarsen (upsample), onto the desired grid.  Since the solver is accurate, it approximates the PDE solution.   We performed a consistency check that the accuracy of the methods was the same, whether we used the benchmark data or generated data.  This method was used to study accuracy of solvers when we varied the coefficients (higher variability), when we consider solutions generated by a different distribution, and when we considered new PDE terms. 

\subsection{Learning neural PDE solvers}
The neural PDE solver corresponds to a neural network architecture, with weights $W$.  
A forward pass (fixed $W$),  maps initial grid data $U_0$, to a vector of time slices $U(X,T)$.
\begin{equation}
\label{NNSolve}
	U(X,T) = \NNOperator(U_0; W)
\end{equation}
  The neural network architecture is designed to be amenable to fitting PDE solutions, however it does not encode a PDE solution operator.   (We study the FNO solver bias in \autoref{subsec:FNOerrors}.) 
  The neural network learns the solution operator by fitting the data, using mean squared (or some other appropriate) loss,
\begin{equation}
\widehat W = \arg	\min_W \sum_{U_i \in S} \| U_i - \NNOperator(U_i(X,0); W) \|_X^2
\end{equation} 
Once the neural network is trained, then the final weights $\widehat W$ are used for a forward pass, in the form of $\NNOperator(U_i(X,0); \widehat W)$,  which corresponds to the approximate solution operator.

%\paragraph{Measuring the bias of neural solvers}
%There may also be a bias in the solution operator, coming from the choice of architecture.   In the results section we show that there is a bias in neural network solvers. 

\section{Learned numerical PDE solvers}
%with given coefficients $a,b$, on a fixed, positive time interval, along with compatible boundary conditions.  

The numerical PDE  solver takes advantage of existing neural network infrastructure to implement a solver for a family of PDEs in the form of  a feed-forward convolutional neural network.  The parameters of the neural networks correspond to the coefficients of each of the linear and nonlinear terms of the PDE operator.

The parameters of the model correspond to grid values of the PDE parameters,  $a(x),b(x)$.  
When the parameters are fixed, the solver takes as input the initial value, $u(x, 0)=u_0(x)$, on a fixed grid, and outputs an approximate finite difference  solution $u(x, t)$, at specified time slices. 
 
Learning the parameters from a dataset of PDEs solutions, taken at time slices and on grid values, corresponds  to learning the coefficients of a numerical PDE solution operator.  Thus it takes as input a dataset of time slices of solutions of \eqref{PDE}, and learns grid values of PDE coefficients $a(x), b(x)$ for the solver. 
 In this case, training corresponds to learning which terms are active in the PDE, and what are the values of the coefficients of those terms.  In many cases the correct parameters are identified. See \autoref{learn_coefs}.
 However, since learning coefficients is in general a harder problem than learning solutions operators, \cite{isakov2006inverse} there are cases where an accurate solution operator may learn equivalent solution operators with different coefficients. 

Stated differently, the parametric PDE learning problem, (which is a type of PDE inverse problem) corresponds to the following. The input data set is assumed to be a solution of \eqref{PDE} with unknown but bounded coefficients.    A given benchmark problem would have most of the coefficients set to zero.  However, each training run assumes all the coefficients can be non-zero, unless otherwise specified. 

DeepFDM is implemented as a feedforward convolutional neural network, where each forward pass corresponds to an implementation  of a scientific computing solver, ${\NumPDESolve}(U; A(X))$, of \eqref{PDE}, where $A(X)$ corresponds to the unknown vector of coefficients.   In other words, for a given vector of coefficients, a forward pass with be a numerical solution of the corresponding PDE.  The numerical PDE operator is implemented using stable finite differences.  It is parameterized by the grid values, $A(X)$, of the coefficients $a(x)$ of the PDE.
The numerical solution operator is written,  
\begin{equation}
\label{NumPDE}
	U(X,T) = {\NumPDESolve}(U_0; A(X))
\end{equation}

Given the dataset $S^m$, training the DeepFDM network corresponds to fitting the coefficients to the data.
\begin{equation}
\label{PDE_Inverse}
	\widehat A = \arg\min_{A \in \mathcal  A} \sum_{U_i \in S} \| U_i - {\NumPDESolve}(U_i(X,0); A) \|_X^2
\end{equation} 

Using machine learning terminology, the method \eqref{PDE_Inverse} is \emph{interpretable by design}: the parameters of the model correspond to the coefficients of the PDE. 

Since DeepFDM corresponds to learning parameters of the PDE, we expect that overfitting will not be a problem, and since a forward pass corresponds to an accurate numerical solution of the PDE, we expect that it will be accurate.  However, there will still be errors associated with finite data, so we expect $\widehat A$ to approximate, but not be equal, to $A^*$, the grid values of the coefficients. The problem \eqref{PDE_Inverse} corresponds to the parametric vector regression problem. The problem corresponds to \emph{parameter identification}.  This is more tractable and easier to analyze than  the neural network problem~\eqref{NNSolve}. 

In the next section, we have a theorem which characterizes the method, in terms of parametric regression.  In a later section, we show that \eqref{NNSolve} can be biased.

\section{Parametric Regression problem for the numerical PDE solver}

For a given grid, $X$, let $h(X)$ be the grid resolution. 
Regarding a function on a grid as an approximation, we define $\| U \|_X = h(X)^2 \|U\|_2$ (in two dimensions).   The scaling factor is a normalization that ensures constant functions have the same norm, regardless of the grid resolution. 
 
\begin{theorem}
Let $A^*$ be the grid values of the true PDE parameters, $A^* = a(X)$.  
The numerical PDE learning problem corresponds to  a vector regression parametric learning problem  	
\[
\min_{A \in \mathcal A } \sum_{U_i \in S} \| {\NumPDESolve}(U_i(X,0); A) -  {\NumPDESolve}(U_i(X,0); A^*)- \epsilon_i \|_X^2
\]
with noise vector, $\epsilon_i$,  whose norm goes to zero with the grid resolution,
 $\max_i \|\epsilon_i\|_X  =  \mathcal O(h(X))$. 
\end{theorem}

\begin{proof}
Given the \eqref{PDE}, write $u(x,t) = \text{PDESolve}(u_0, a(x))$ for the solution of the PDE, with initial data $u_0(x)$.

Define $$ \epsilon_i = {\NumPDESolve}(U_0; A(X)) - \text{PDESolve}(u_0, a(x))(X)$$ 

Standard PDE finite difference numerical approximation bounds \cite{Stigg} can then be expressed as  	$\|\epsilon_i   \|_{X}^2 = \mathcal O (h(X))$, where we assume first order accuracy.   The PDE solution $u(x,t)$ when evaluated on the grid, corresponds to $U_i$. 
Thus we have 
 satisfies $U_i - {\NumPDESolve}(U_i(X,0); A^*) = \epsilon_i$, where $\epsilon_i$ represents the numerical solver error, which has norm on the order of the grid resolution $\|\epsilon_i\|_X  =  \mathcal O(h(X))$, as desired. 
\end{proof}

The small amount of noise means there can be a small error in learning the parameters, but still we expect that the model learn a close approximation of the correct parameters, and should generalize.  Thus, using standard results about regression, this theorem tells us that we expect a nearly unbiased approximation to the true paramters of the model, with better results as the grid resolution improves. In many cases, for inverse problem, there is theory that ensures machine learning \emph{consistency:} with enough data the solution operator converges to the correct one.  With additional assumptions, the coefficients also converge,~$\widehat A \to A^*$.

\section{DeepFDM model architecture}

\begin{figure}
    \centering
    \includegraphics[height=0.5\textwidth]{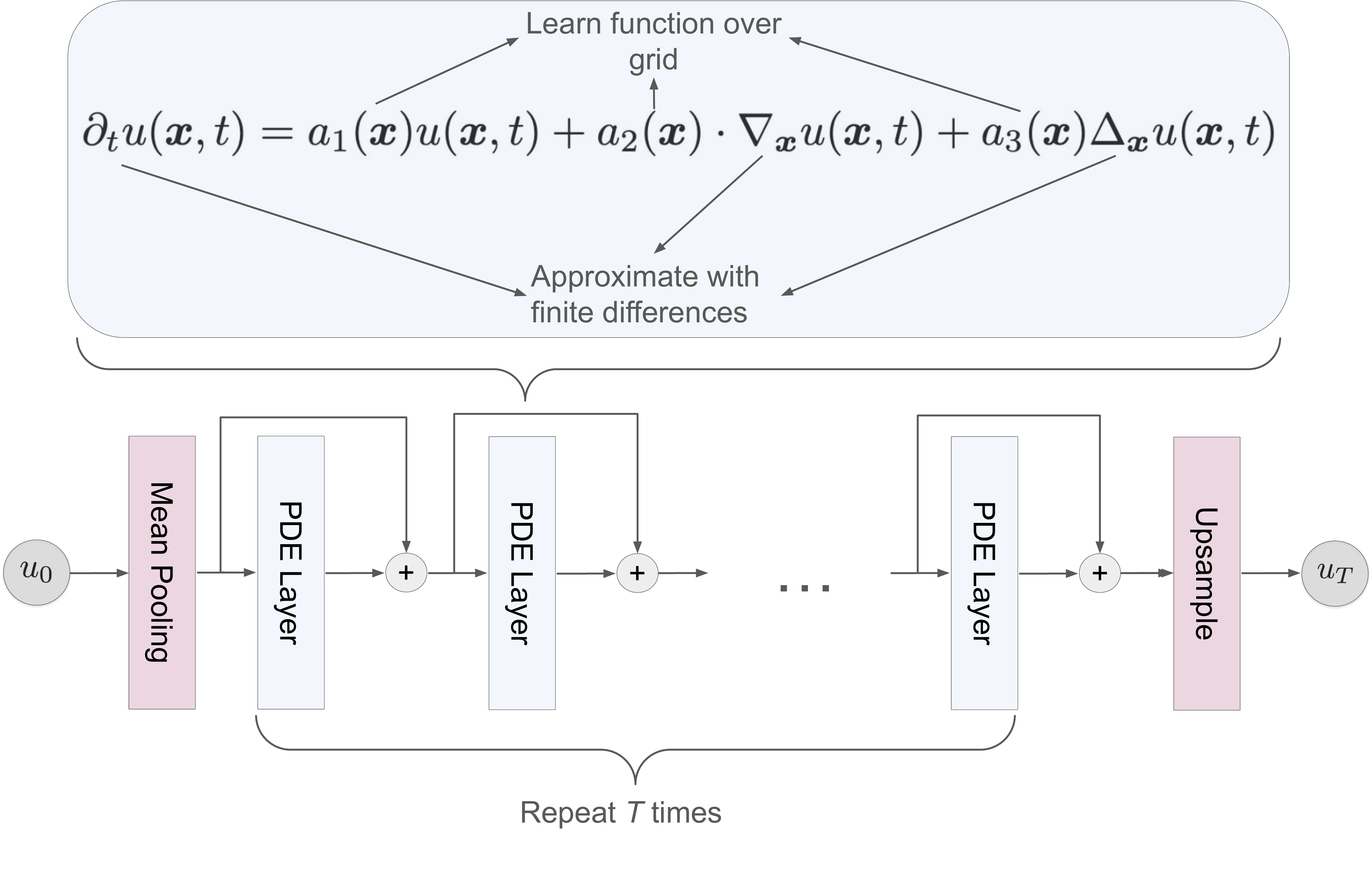}
    \caption{Example network architecture for a linear nonconstant PDE over space and time. The mean pooling layer is used to reduce the resolution of the input and the upsampling layer is used to bring the output back up to size.}
  \label{ArchitectureDiagram}
\end{figure}

DeepFDM can be interpreted as neural network architecture, which implements a stable, consistent finite difference method for solving a parametric family of PDEs, \eqref{PDEcoeffs} as a forward pass.  Training the model corresponds to learning the relevant terms of the PDE, as well as the corresponding coefficients.

DeepFDM is implemented as a feedworward convolutional neural network, where each forward pass corresponds to an implementation  of a scientific computing solver, ${\NumPDESolve}(U; A(X))$, of \eqref{PDE}, where $A(X)$ corresponds to the unknown vector of coefficients.   In other words, for a given vector of coefficients, a forward pass is a numerical solution of the corresponding PDE.  The numerical PDE operator is implemented using finite differences, and the PDE is parameterized by the grid values, $A(X)$, of the coefficients $a(x)$.
The details of the implementation and the use of finite difference schemes are given in \autoref{sec:FD},  for reference. 

Using a neural network architecture is very convenient, since we can take advantage of built-in optimization routines, rather than implementing optimization as is more typical with inverse problems.  Moreover,   a forward pass is very computationally efficient, since we are working with a deep but small convolutional neural network.    Training is also faster than for benchmark neural solvers, see the results section below.

The model architecture is an implementation of the finite difference solver for the PDE \eqref{PDE}.  In this case, the finite difference operators are implemented as convolution with fixed (predefined, non-learnable) operators.  The coefficients correspond to model weights passed through a sigmoid nonlinearity (to make them bounded).  
This allows the finite difference solver,  to be implemented as a differentiable model, and trained using standard SGD implementations (see the results section below). 
The architecture is illustrated in \autoref{ArchitectureDiagram}, (where the PDE is written using different notation), more details of the architecture can be found in \autoref{sec:arc2}.

\subsection{Inverse problems with a convergent finite difference solver}\label{sec:FD}
 In this section, we demonstrate the finite difference operator in a simple case and give an idea of how to build learnable finite difference operators in the general case.

A fundamental result in numerical approximation of linear PDEs~\cite{courant1928partiellen,Stigg} provides conditions on the time and grid discretization parameters,  $ {c_t},{c_x}$, in terms of bounds on the coefficients $a(\vx)$ which ensure that the method is numerically stable, and convergent. 
\cite{oberman2006convergent}, extended the family of stable finite difference operators to a wide class of diffusion-dominated PDEs.
The convergence theory states that
as the resolution of the data increases, the solution operator converges to the PDE solution operator. 
$\lim_{\epsilon\to 0} \|h^{\epsilon} - h^*\| = 0$,  in the appropriate operator norm.

\subsubsection{Consistent finite difference operators on grids}
A finite difference operator on a grid is an approximation of a differential operator. The finite difference operator corresponding to $u_{\vx}$ on a uniform grid is represented by a convolution operator, with kernel $W = \frac{1}{c_x}[-1,1]$,  where we have replaced $\epsilon$ with the grid spacing parameter, $c_x$.  We have a similar operator in two dimensions. 

Finite difference approximations of the Laplacian   in one dimension,  in two dimensions, correspond, respectively, to convolutions with the following kernels
\begin{equation}
    \label{kernels}
W_{Lap}^{d=1} = \frac{1}{c_x^2}[1,-2,1],
\qquad 
W_{Lap}^{d=2} = 
\frac{1}{c_x^2}\begin{bmatrix}
0 & 1 & 0\\
1 & -4 & 1\\
0 & 1 & 0
\end{bmatrix}.
\end{equation}
 These operators are linearly combined to approximate each of the linear terms, $L(u,a)$, in the linear part of the PDE.
 
 To approximate nonlinear terms, 
 we use upwind nonlinear finite difference operators, \cite{oberman2006convergent}, which showed that 
it is possible to build \emph{numerically stable} finite difference approximations for a wide class of nonlinear elliptic and parabolic PDEs. (For example, a stable approximation of the  the eikonal operator, $|u_x|$, is given by the maximum of the upwind finite difference schemes for $u_x$ and $-u_x$, respectively.)

%Thus, for all parameters, $\theta$, the model  corresponds to a convergent discretization of some PDE, $P(u,a)$, for some $a \in \mathcal A$.

\subsection{Stable, monotone discretizations}\label{modeldeets}
Each layer of the operator corresponds to a discretization of a PDE.  We need this discretization to be convergent, which puts requirements on the hyperparameters in the model, and how they relate to the possible coefficients. Here we discuss the special case of the heat equation, for clarity of exposition.  

When solving any PDE numerically, we are bound by some stability constraints that are necessary for obtaining a convergent solution. For the heat equation, assuming we take space intervals of $c_x$ (and equal in all dimensions) and time intervals of $c_t$, we are bound by the stability constraint $0\leq a(\vx)\cdot \frac{c_t}{c_x^2}\leq \frac{1}{2\cdot D}$ where $D$ is the dimension of the data,~\cite{courant1967partial}.  Thus when one knows the coefficients $a(\vx)$ then one can simply pick $c_t$ and $c_x$ to satisfy the stability constraint.

In this case, we take the opposite approach. Given fixed values of $c_x$ and $c_t$, we can bound the coefficients themselves by 

\begin{equation}\label{boundcoefs}
\tag{CFL}
    0  \leq a(\vx)\leq C_{a} = \frac{c_x^2}{2D\cdot c_t}
\end{equation}

This is a crucial constraint since the parameters of the model will take the place of the coefficients of the equation being modelled. In this way, we design DeepFDM precisely with the aim of learning the physical process that is trying to approximate. 

In order to satisfy the stability constraint, we bound the raw parameters learned by the model with a scaled sigmoid function. This is, if the model's parameters are $\theta$, then the values that we multiply with the convolution layer (corresponding to the Laplace operator) are given by $C_a \cdot \sigma(\theta)$. This ensures that the parameters are bounded by the stability region of the PDE and thus forces DeepFDM to find a solution in the parameter space in which the PDE itself is stable.

\subsection{Example: finite difference heat equation solver}
Here we give the simple example of the one dimensional heat equation, which will be used to illustrate the architecture of the stable solver.  Familiar readers can skip to the general case below.

For $\vx$ in one dimension, use the finite difference approximations, 
$$u_{\vx}(\vx)  \approx \frac{u(\vx + \epsilon)-u(\vx)}{\epsilon}$$
along with 
$$u_{\vx \vx}(\vx)  \approx \frac{u(\vx + \epsilon)-2u(\vx)+ u(\vx - \epsilon)}{\epsilon^2}.$$
We apply this finite difference operator to approximate  the heat equation, 
$ 
u_t = a(x) u_{xx}.
$ 
The forward Euler finite difference method for the heat equation, is \cite{Stigg}: 
%We can perform linear interpolation to obtain off grid values to approximate $u(x,t)$. 
\begin{equation}\label{euler}
    U_{j,k+1}  = U_{j,k} 
    + \frac{c_t}{c_x^2} A_j\left( U_{j+1,k} -2U_{j,k} + U_{j-1,k}\right),  
\end{equation}
The full solution operator, at time $t = n c_t$ is obtained by iterating the operator above $n$ times. 
This operator is consistent, it is also stable, provided the CFL condition, \eqref{boundcoefs}, below, holds. 
Below we also include an additional sigmoid nonlinearity applied to the coefficient $A_j$, which ensures that it is bounded for all values of $A_j$.

\subsection{Neural network architecture}\label{sec:arc2}

DeepFDM is formulated to take some initial condition $U_0$ (defined as a function over a grid) and iterate it forward in time for some given number of time steps $T$. For $k = 0, 1, \hdots, T-1$, the iterative update is defined as 
\begin{equation}
	U_{k+1} :=
U_{k} + c_{t} \left (\sum_{i=1}^{n_{A}} \sigma (A_i) \odot \left({\text{conv}(W_i,U_k)}\right) 
+ \sum_{j=1}^{n_B} N_j (B_j,U_k,U_{k-1})
\right) 
\end{equation}
where $\odot$ represents the component-wise product. Here $c_t$ is a constant representing the time step interval, and each $W_i$ is a predetermined (non-trainable) convolution kernel corresponding to a finite difference operator.  These kernels were illustrated in \autoref{kernels}.  The component-wise  $\sigma (\cdot )$ is the sigmoid function. $n_{A} =4 $ and $n_{B} =2$ represent the number of linear and nonlinear terms, respectively.  The grid vectors $A_i, B_j$ are the parameters corresponding to each linear and nonlinear term, respectively.   The linear terms corresponds to standard upwind finite difference discretizations of the derivatives, pointwise multiplied by nonlinearly scaled coefficient terms.  

The $N_j$ represent nonlinear terms. For the reaction operator, for example, this corresponds to a quadratic reaction term and takes the form
$$
N_1(B_1, U) = \sigma(B_1) \odot U \odot  (1-U),
$$ 
which is component wise multiplication. 
The nonlinear term corresponding to the non-constant coefficient  Burgers' operator takes the form
$$
N_2(B_2,U_k,U_{k-1}) =  \sigma(B_2)\odot 
 U_{k-1}\left({\text{conv}(W_{adv},U_k)}\right)$$
where the first term corresponds to a linear advection term, multiplied by $U_{k-1}$. The previous time step value for computing the spatial gradient was used for stability.

By bounding the coefficients with a sigmoid function scaled by the time step interval, our model corresponds by design, for fixed parameter values $\theta$, to a stable finite difference method consistent with a PDE  with coefficients given by the model parameters.  

More PDE terms can be added, for example we experimented with eikonal type terms.  

Each layer is repeated $T$ times, and corresponds to the Forward Euler method, as can be seen in \autoref{ArchitectureDiagram}.

\section{Dataset generation and OOD quantification}\label{sec:data}
% We describe the procedure used to generate synthetic data used for testing
% We also describe what we mean by OOD shift. To the best of our knowledge, no other work employs the OOD quantification we use here making this also a novel contribution.
% Finally we discuss choice of benchmarks 
In this section, we describe the procedure used to generate the synthetic data we use for testing. We equally explain the notion of OOD shift we consider in this paper. To the best of our knowledge, no other work employs the OOD quantification scheme used in our paper, making it a novel contribution.

\paragraph{Numerical PDE solvers}\label{sec.PDE.solvers}
We implemented a finite difference solver~\cite{Stigg}, which could solve any PDE of form \eqref{PDE}, provided bounds on the coefficients.  The PDE solver uses the forward Euler method, with a small time step, which is calculated from the coefficients and spatial resolution, to ensure stability and convergence.   This solver can be used to generate high resolution PDE solutions, with a given distribution of coefficients.  High resolution PDE solutions were projected onto a coarser grid, which is a standard method for generating approximate PDE solutions.

\paragraph{Data generation process}
To characterize and benchmark DeepFDM against existing architectures, we train on synthetic data generated by PDE solvers. 
\begin{enumerate}
    \item Sample some Fourier coefficients $c \sim \mathcal{N}(0, \Sigma)$ from a Fourier spectrum with at most $N$ modes and compute the resulting function, $U_0$,  with coefficients multiplied by the fourier basis functions. 
    \item Use a standard scientific computing solver to compute the solution to the PDE problem with initial condition $U_0$ for the required number of time steps.
\end{enumerate}

%For instance, in 2D, the initial conditions would be generated following
%\begin{equation}
%    \label{data_gen1}
%u(\vx,0) = \sum_{i=1}^{N}\sum_{j=1}^N c_{ij} \sin(\pi ix)\sin(\pi j y),
%\qquad 
%c_{ij} \sim  \mathcal{N}(0,\Sigma),
%\end{equation}
%where $\mathcal{N}(0,\Sigma)$ is a mean zero normal distribution with diagonal covariance matrix. For the purpose of this work, we use $N=4$ Fourier modes for all problems. To generate OOD data, we employ the same procedure, using a different distribution $\Tilde\rho = \mathcal{N}(0, \Tilde\Sigma)$. 

%\autoref{Fourier} shows an example (simplified for visualization) of the different Fourier spectra used to generate different data distributions and 
\autoref{1dhellinger} shows examples of  initial conditions generated using different Fourier spectra. 
Figures for one-dimensional data are shown for illustration purposes only, reported results are for two dimensions.

\paragraph{Measuring dataset shift }
We measure the distance between data generating distributions using the Hellinger Distance.
The Hellinger distance between two multivariate, mean zero, normal distributions is given by \cite{cramer1999mathematical}
\[
H^2(\mathcal{N}(0, \Sigma),\mathcal{N}(0, \Tilde\Sigma))
      =1-
      \frac{{\det(\Sigma)^\frac{1}{4}\det(\Tilde\Sigma)^\frac{1}{4}}}
{{\det\left({(\Sigma+\Tilde\Sigma)}/{2}\right)^\frac{1}{2}}}.
\]
\autoref{1dhellinger} shows the Hellinger distances between the distributions which generated the samples.

\begin{figure}
    \centering
    \begin{subfigure}[b]{0.22\textwidth}
        \centering
        \includegraphics[width=\linewidth]{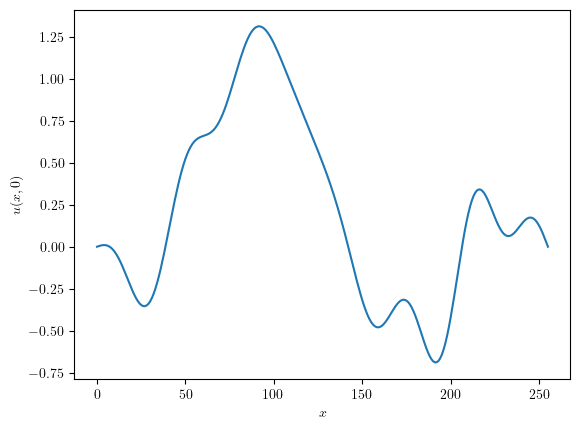}
        \caption{one dimensional training distribution sample}
        \label{1dIDH}
    \end{subfigure}%
    \hfill
    \begin{subfigure}[b]{0.22\textwidth}
        \centering
        \includegraphics[width=\linewidth]{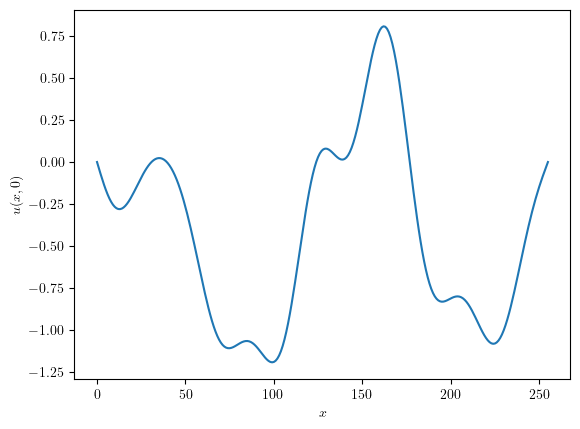}
        \caption{Sample from a distribution with $H^2=0$}
        \label{1d0}
    \end{subfigure}\quad
    \hfill
    \begin{subfigure}[b]{0.22\textwidth}
        \centering
        \includegraphics[width=\linewidth]{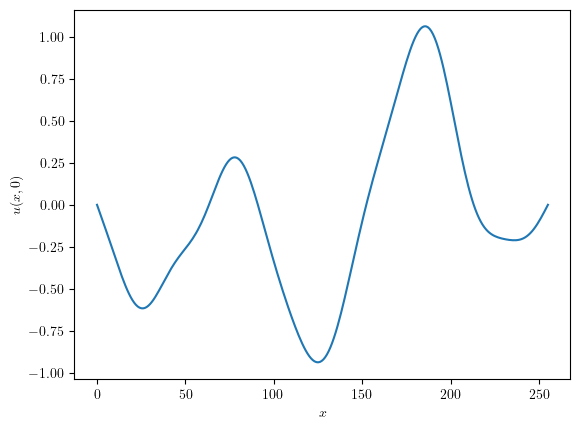}
        \caption{Sample from a distribution with $H^2=0.64$}
        \label{1d0.8}
    \end{subfigure}
    \hfill
    \begin{subfigure}[b]{0.22\textwidth}
        \centering
        \includegraphics[width=\linewidth]{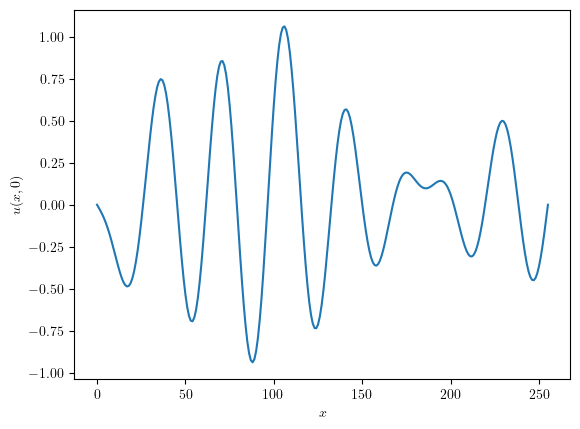}
        \caption{Sample from a distribution with $H^2=0.99$}
        \label{1d0.99}
    \end{subfigure}

    \medskip
    \begin{subfigure}[b]{0.22\textwidth}
        \centering
        \includegraphics[height=1in]{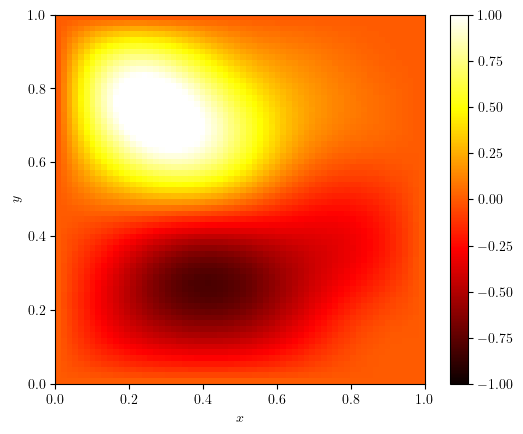}
        \caption{two dimensional training distribution sample}
        \label{2did}
    \end{subfigure}%
    \hfill
    \begin{subfigure}[b]{0.22\textwidth}
        \centering
        \includegraphics[height=1in]{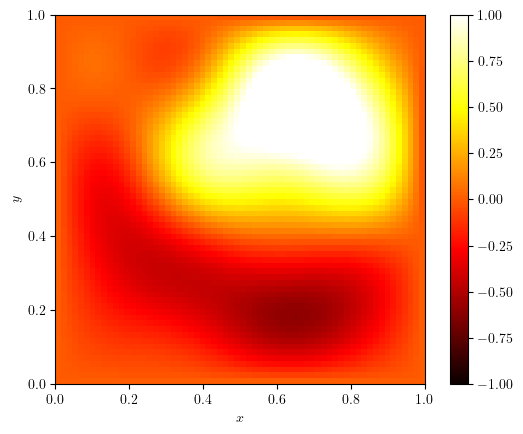}
        \caption{Sample from a distribution with $H^2=0$}
        \label{2d0}
    \end{subfigure}%
    \hfill
    \begin{subfigure}[b]{0.22\textwidth}
        \centering
        \includegraphics[height=1in]{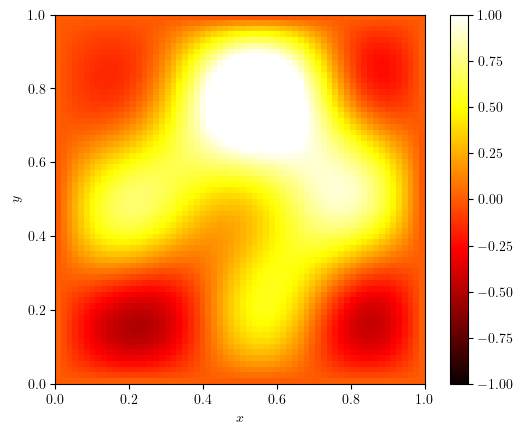}
        \caption{Sample from a distribution with $H^2=0.64$}
        \label{2d0.8}
    \end{subfigure}
    \hfill
    \begin{subfigure}[b]{0.22\textwidth}
        \centering
        \includegraphics[height=1in]{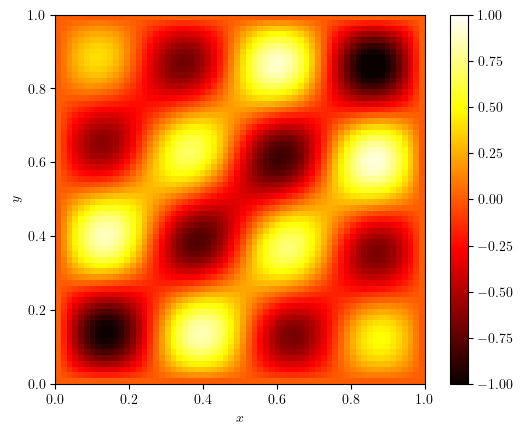}
        \caption{Sample from a distribution with $H^2=0.99$}
        \label{2d0.99}
    \end{subfigure}
  \caption{Examples of training data and test data from different distributions, measured through the Hellinger distance.  Top: one dimensional.  Bottom: two dimensional data.}
  \label{1dhellinger}
\end{figure}

\section{Summary of the main results}
%\subsection{Neural Network PDE solvers}
The primary neural network solver benchmark is the Fourier Neural Operator (FNO), from  \cite{Li:2020aa}. This well-known method, has been used as standard benchmark for other neural solver approaches.  We also implement two other methods, for comparison, which illustrates the advantage of the FNO approach over other neural network approaches.

\subsection{Out-of-distribution data} 
We also introduce a principled method for measuring out-of-distribution (OOD) performance of neural operators.  We generate PDE solution data by choosing families of orthogonal functions  with random coefficients for the initial conditions, and passing these functions through a high accuracy numerical PDE solver.  Using different distributions on the coefficients allows us to quantify the distance between the distribution (over functions), in terms of the Hellinger distance (other divergences could be used, this one was convenient for implementation).  These results are illustrated in~\autoref{2dFit}, and quantified in~\autoref{OODfinal}.

We compare the models, using both the PDE solvers metrics and and neural network metrics. These are accuracy in and out-of-distribution, model size (number of parameters), and model training time (measured in epochs).  Our results show that the numerical PDE approach outperforms state of the art neural solvers.  Summary results are presented in  \autoref{tab:model_comparison}.  Details can be found in \autoref{sec:numresults}, below. 

DeepFDM which is based on a parametric PDE solver, has a performance improvement of an order of magnitude better on all metrics, compared to FNO, which is a PDE compatible neural network (FNO outperforms other neural network architectures). Accuracy is expressed an a range, over three linear and two nonlinear PDEs.  Out-of-distribution (OOD) accuracy: uses the accuracy metric as before but over a range of different data distributions. Because the distributions differ, the solutions have been normalized for a fair comparison across distributions.   Parameter count is the number of parameters of the model listed, in the case of a data dimension of 4096, corresponding to the number of grid points.  Training time is the (approximate) number of epochs to reach training loss of $5\times 10^{-4}$.   The different PDEs and corresponding accuracy for each one are provided in \autoref{acc}. % and \autoref{tab:method_comparison}.

\begin{table}[h]
\centering
\begin{tabular}{|l|c|c|c|c|}
\hline
\textbf{Model} & \textbf{Accuracy} & \textbf{Accuracy (OOD)} & \textbf{Params } & \textbf{Train Time} \\
\hline
FNO      & 0.02--0.04     & 0.03--0.3       & 184,666  & 120 \\ \hline
DeepFDM  & 0.001--0.004   & 0.0007--0.009   & 20,484   & 10  \\  \hline
Ratio    & 10X--20X       & 30X--40X        & 9X       & 12X \\ 
\hline
\end{tabular}
\caption{Summary comparison between the neural operator approach, FNO, and the PDE discretization approach, DeepFDM (smaller numbers are better in all columns). DeepFDM which is based on a parametric PDE solver, achieves an order of magnitude improvement on all metrics, compared to FNO, which is uses a PDE compatible neural network (FNO outperforms other neural network architectures). Accuracy is expressed in a range over three linear and two nonlinear PDEs.  OOD accuracy: is the same accuracy metric as before but over different data distributions, in addition, because the distributions differ, the solutions are normalized.  Parameter count is for the model with data dimension (number of grid points) of 4096.  Training time is number of epochs to reach training loss of $5\times 10^{-4}$.
}
\label{tab:model_comparison}
\end{table}

The trade-off for using the numerical PDE approach is that broadening the class of PDEs  requires modifications to the architecture. For example, adding different boundary conditions, changing the dimension of the space, or adding new nonlinear terms would require the expertise to implement a finite difference discretization of the terms.  While this is standard, it requires a domain expertise which is not required for  the neural operator approach.   On the other hand the neural operator approach performs worse, and has no guarantees that the solution operator is compatible with a PDE solution operator.  In fact basic tests, such as the linear superposition property fail for the neural operator approach.

\begin{table}[h]
    \centering
    \begin{tabular}{|l|c|c|c|c|}
        \hline
        PDE   & ResNet & U-Net & FNO & DeepFDM\\
         \hline
        Diffusion equation & 0.6149 & 0.0640 & 0.0266 & \textbf{0.0024} \\
        Advection equation & 0.7039 & 0.0618 & 0.0251 & \textbf{0.0007}\\
        Advection-diffusion & 0.6286 & 0.0692 & 0.0307 & \textbf{0.0017}\\
        Reaction-diffusion& 0.9119 & 0.0521 & 0.0319 & \textbf{0.0016}\\
        Burgers' equation & 0.8517 & 0.0790 & 0.0379 & \textbf{0.0045}\\
        \hline
    \end{tabular}
    \caption{Test error (normalized MSE) of our model and various benchmarks on a different PDEs.  The results reported are the average over three training runs. The PDE solver DeepFDM method outperforms neural operators on accuracy by approximately one order of magnitude.    
    }
    \label{acc}
\end{table}

\begin{figure}
    \centering
    \begin{subfigure}[b]{0.5\textwidth}
        \centering
        \includegraphics[height=2.in]{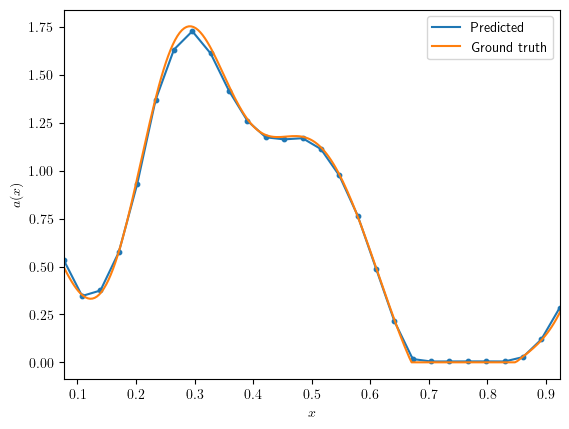}
        \caption{1D diffusion process}
        \label{1dcoefs}
    \end{subfigure}%
    \hfill
    
    \medskip
    \begin{subfigure}[b]{0.45\textwidth}
        \centering
        \includegraphics[height=2.in]{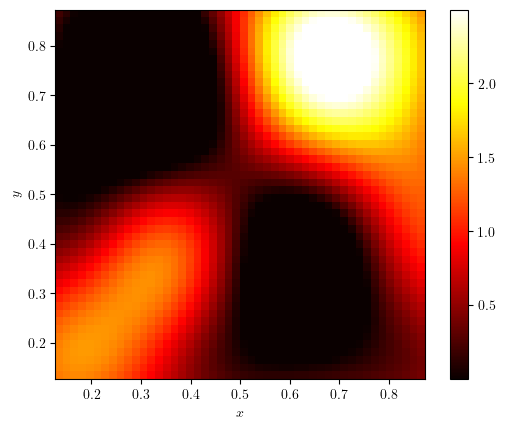}
        \caption{2D ground truth coefficients}
        \label{diffact}
    \end{subfigure}%
    \hfill
    \begin{subfigure}[b]{0.45\textwidth}
        \centering
        \includegraphics[height=2.in]{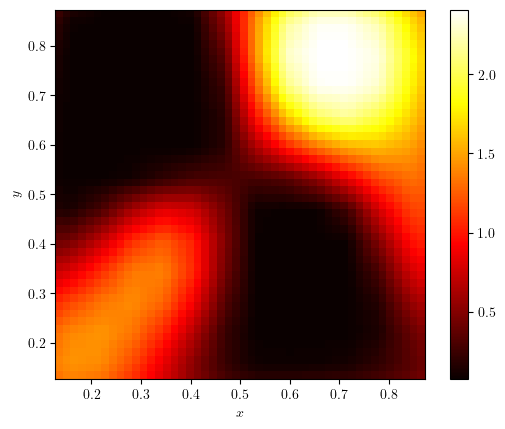}
        \caption{2D learned coefficients}
        \label{difflearn}
    \end{subfigure}%
    \hfill
  \caption{DeepFDM learns a convergent finite difference discretization of a PDE, with unknown  terms and coefficients corresponding. (a) Exact and learned coefficients for a one dimensional heat equation. Exact (b), and learned (c) coefficients of a two dimensional heat equation.}
  \label{learn_coefs}
\end{figure}

\begin{figure}
    \centering
    \begin{subfigure}[b]{0.24\textwidth}
        \centering
        \includegraphics[height=0.9in]{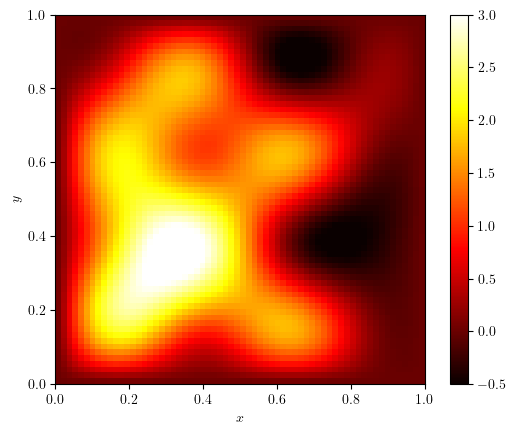}
        \caption{ID ground truth}
        \label{2Dsolexact}
    \end{subfigure}%
    \hfill
    \begin{subfigure}[b]{0.24\textwidth}
        \centering
        \includegraphics[height=0.9in]{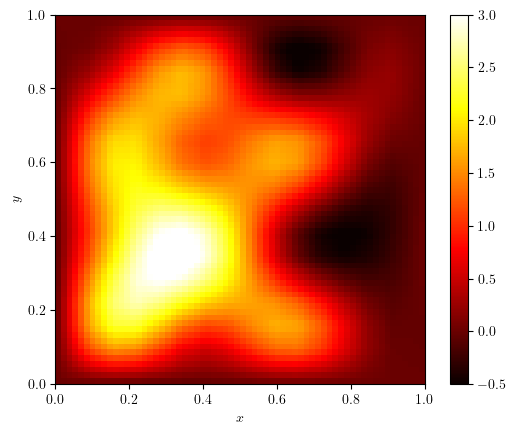}
        \caption{ID DeepFDM}
        \label{2DsolOurs}
    \end{subfigure}%
    \hfill
    \begin{subfigure}[b]{0.24\textwidth}
        \centering
        \includegraphics[height=0.9in]{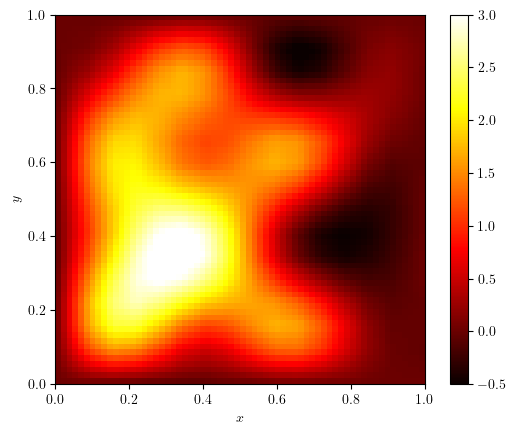}
        \caption{ID FNO}
        \label{2DsolNN}
    \end{subfigure}
    \hfill
    \begin{subfigure}[b]{0.24\textwidth}
        \centering
        \includegraphics[height=0.9in]{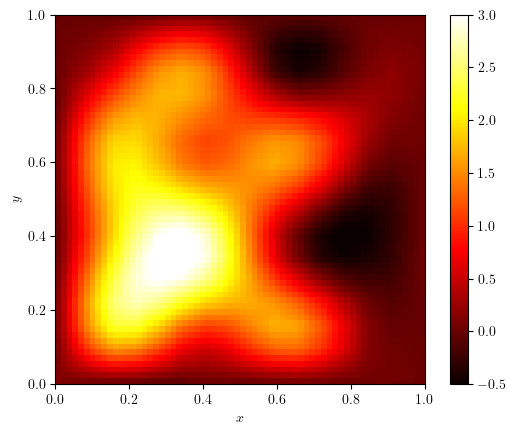}
        \caption{ID U-Net}
        \label{2DsolCNN}
    \end{subfigure}

    \medskip
    \begin{subfigure}[b]{0.24\textwidth}
        \centering
        \includegraphics[height=0.9in]{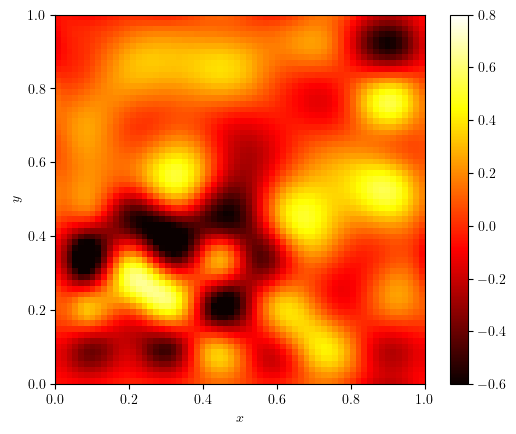}
        \caption{OOD ground truth}
        \label{2DsolexactOOD}
    \end{subfigure}%
    \hfill
    \begin{subfigure}[b]{0.24\textwidth}
        \centering
        \includegraphics[height=0.9in]{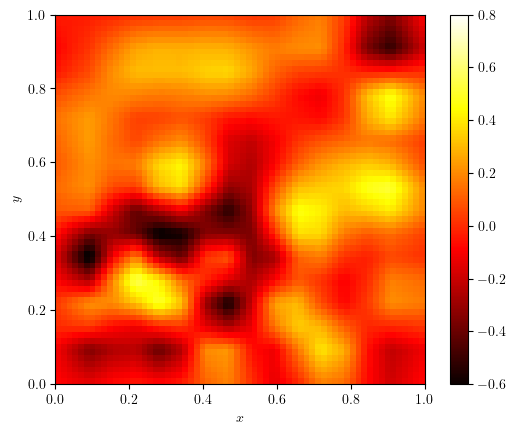}
        \caption{OOD DeepFDM}
        \label{2DsolOursOOD}
    \end{subfigure}%
    \hfill
    \begin{subfigure}[b]{0.24\textwidth}
        \centering
        \includegraphics[height=0.9in]{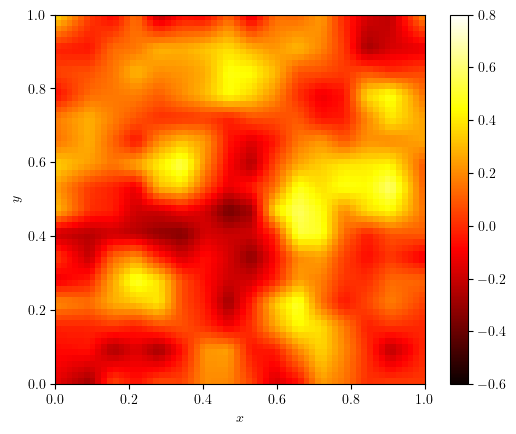}
        \caption{OOD FNO}
        \label{2DsolNNOOD}
    \end{subfigure}
    \hfill
    \begin{subfigure}[b]{0.24\textwidth}
        \centering
        \includegraphics[height=0.9in]{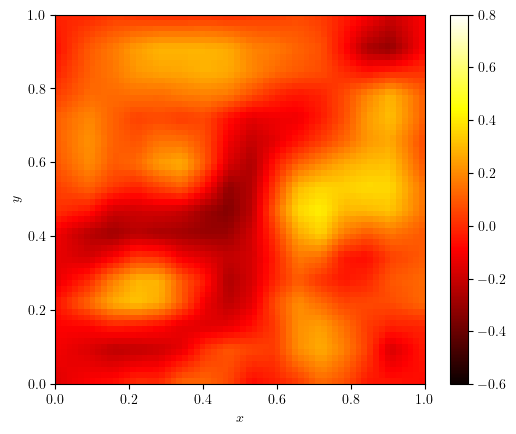}
        \caption{OOD U-Net}
        \label{2DsolCNNOOD}
    \end{subfigure}
  \caption{Two-dimensional solutions for a diffusion equation for both in-distribution (ID) data (top) and out-of-distribution (OOD) data (bottom). All models are visually similar on the in-distribution data. On OOD data, FNO and U-Net lose accuracy.}
  \label{2dFit}
\end{figure}

\section{Numerical results}\label{sec:numresults}
We used datasets generated by~\cite{liu2022predicting} for in-distribution data.  
For Out-of-distribution data, we used synthetic data generated by a numerical PDE solvers, using functions with random Fourier coefficients, as described in \autoref{sec:data}.

\paragraph{Considered benchmarks}
\textbf{U-Net:} We use a U-Net architecture, popular for image-to-image tasks, such as segmentation. We consider a 2D U-Net with 4 blocks~\cite{ronneberger2015u}.
\textbf{ResNet}: We use an 18 block ResNet with residual connections~\cite{he2016deep}.
\textbf{FNO:} We use a FNO with 12 modes for all channels and all experiments~\cite{Li:2020aa}.
All models are trained using the MSE loss function.

\paragraph{Parameter  count and training dynamics}
All models were trained with the Adam optimizer, with no weight decay. Training, validation, and test data samples were split as $75\%$, $12.5\%$, and $12.5\%$ respectively. All models were run on a Tesla T4 GPU with a batch size of $32$.  We can see from \autoref{tranining} that DeepFDM has the smoothest training curve among all models tested. 

The number of parameters in DeepFDM is on the order of the number of grid points (spatial data points) as shown in \autoref{parameters2d}, which is hundreds of times fewer than FNO and U-Net.  

 Training was performed using least squares vector regression, on the data.  DeepFDM trained to MSE $10^{-3}$ in a few epochs, and to $10^{-4}$ in under 100 epochs, which is significantly faster than the other methods. See~\autoref{tranining}.

\begin{table}[h!]
\centering
\begin{tabular}{ |l||rrrr| }
 \hline
Grid resolution   & $4,096$  & $1,024$ &  $256$ & $64$ \\
 \hline
 \hline
 Parameters in DeepFDM   & $20,484$  & $5,124$ &  $1,284$ & $324$\\
 Parameters in  FNO &  $184,666$  &  $184,666$  & $184,666$ & $184,666$\\
 Parameters in  U-Net &  $7,762,762$  &  $7,762,762$  & $7,762,762$ & $7,762,762$\\
 Parameters in  Res-Net &  $3,960$  &  $3,960$  & $3,960$ & $3,960$\\
 \hline
\end{tabular}
\caption{Model parameters for DeepFDM and the benchmark models tested.}
\label{parameters2d}
\end{table}

\subsection{Model accuracy}

We see from \autoref{acc} that, as expected, DeepFDM is more accurate, by a factor of $10$ in all considered equations. Furthermore, as expected, DeepFDM is more accurate in OOD performance; the solutions given by DeepFDM are visually accurate and have lower errors than the other benchmarked models. \autoref{2dFit} shows example modelled solutions in two dimensions. We note that in all Figures we exclude ResNet, since errors were higher than $60\%$.

\subsection{Out-of-distribution generalization and relative error measures}\label{erros}
 To quantify the generalization abilities of DeepFDM to distinct data distributions, we test on several distributions, each one further apart from the training distribution.  Because sampling from different distributions changes the statistics of the solutions, we need to normalize the errors across distributions.  This is done by normalizing each initial function to set the variance of the initial data (as a function of $\vx$) to be one.
Then, the relative errors are calculated as the average error of the predicted solutions (using the normalized initial data). The  is allows us to perform a fair comparison of errors across distributions.

 \autoref{OODfinal} shows the relative error of the models tested as a function of the Hellinger distance between the training and test distribution. We can see that as the test distribution is further apart from the training distribution, all models start losing accuracy but DeepFDM still achieves under $1\%$ relative error while both U-Net and FNO approach errors of $10\%$ for the furthest distributions. We tested this for all the equations shown in \autoref{acc}.

\begin{figure}
    \centering
    \includegraphics[height=2.8in]{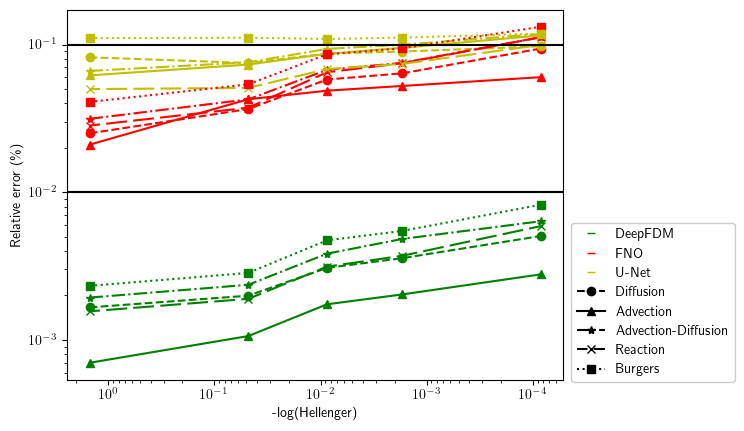}
    \caption{
Relative error of different models (y-axis) in terms of the Hellinger distance between the training and test distributions (x-axis). 
DeepFDM is the most accurate, achieving under $1\%$ relative error even under the largest dataset shift. On the other hand, FNO and U-Net significantly decrease their performance on distinct distributions, with relative errors approaching $10\%$.
    }
  \label{OODfinal}
\end{figure}

\subsection{Interpretable models: learning coefficients}\label{sec:interp}
In most cases, DeepFDM successfully learns a set of parameters that match the ground truth process.
\autoref{learn_coefs}, shows the case of a diffusion equation, where the coefficients are recovered from the model with high accuracy.
%\autoref{learn_swirls}, on the other hand, shows an example of an advection-diffusion equation, where DeepFDM is accurate, but learns a slightly different PDE. This is not a problem with the specification of the model; many sets of coefficients can lead to equivalent solution operators \cite{marchenko2008homogenization}. 
%However, by enforcing a stronger prior, for instance by regularizing for constant coefficient diffusion, this issue could be circumvented.
%

\subsection{Coefficient variance explains FNO errors}\label{subsec:FNOerrors}
In \autoref{coef-variance}, we report relative error for different coefficient values. The coefficient variance on the x-axis corresponds to the amplitudes of the sine waves used to generate the coefficients (larger amplitude corresponding to greater variance). We note that both FNO and U-Net see a degradation in performance as coefficient variance increases while DeepFDM has nearly constant performance.   By design DeepFDM is able to learn variable coefficient PDEs accurately. 

For the FNO, this can be explained by one of the underlying hypothesis of their model architecture; in order to perform computation in Fourier space, the authors make the assumption that the Green's function they learn is \textit{translation invariant}. Since variable coefficients are not translation invariant, as the variance of the coefficients grows, this hypothesis becomes less valid.  Thus we illustrated the bias of FNO towards translation invariant solutions. 

 \begin{figure}
     \centering
     \includegraphics[height=2.25in]{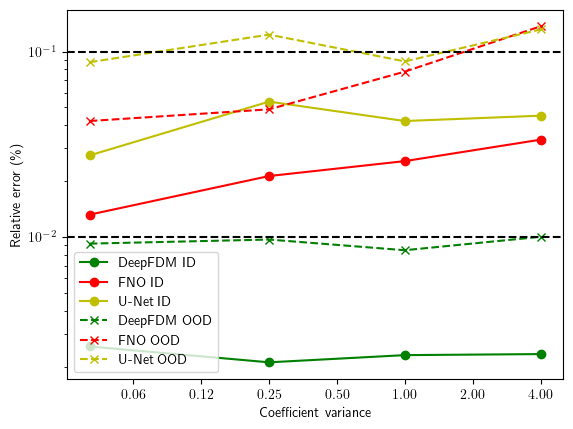}
     \caption{
 Relative error as a function of coefficient variance. The coefficient variance corresponds to the amplitude of the sine waves used to generate the coefficients (bigger variance means bigger amplitude). Our model shows constant error across coefficient size while FNO and U-Net see a performance drop as variance increases.
     }
     \label{coef-variance}
 \end{figure}

\begin{figure}
    \centering
    \begin{subfigure}[b]{0.48\textwidth}
        \centering
        \includegraphics[width=\linewidth]{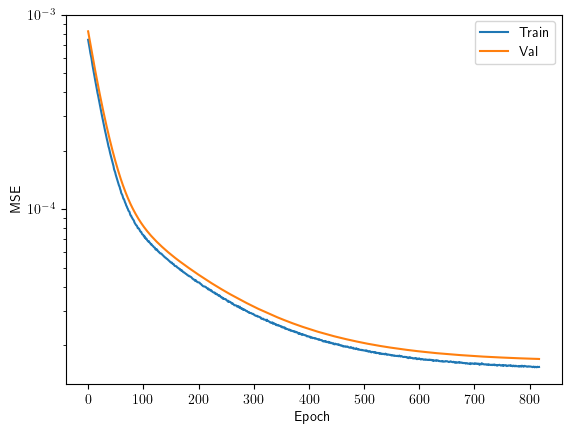}
        \caption{DeepFDM}
        \label{train ours}
    \end{subfigure}%
    \hfill
    \begin{subfigure}[b]{0.48\textwidth}
        \centering
        \includegraphics[width=\linewidth]{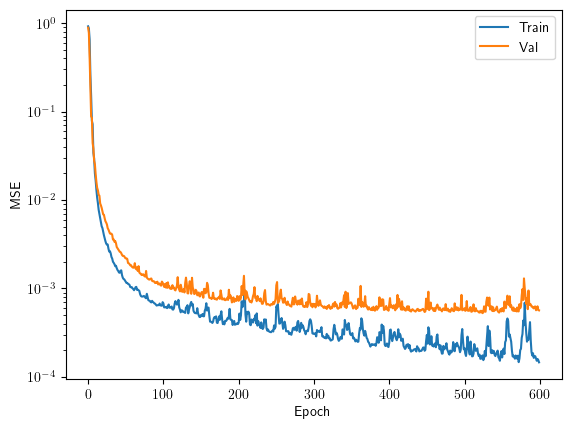}
        \caption{FNO}
        \label{train fno}
    \end{subfigure}

    \medskip
    \begin{subfigure}[b]{0.48\textwidth}
        \centering
        \includegraphics[width=\linewidth]{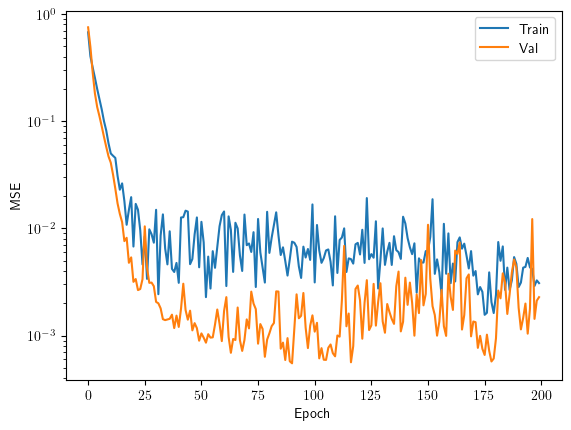}
        \caption{U-Net}
        \label{train unet}
    \end{subfigure}%
    \hfill
    \begin{subfigure}[b]{0.48\textwidth}
        \centering
        \includegraphics[width=\linewidth]{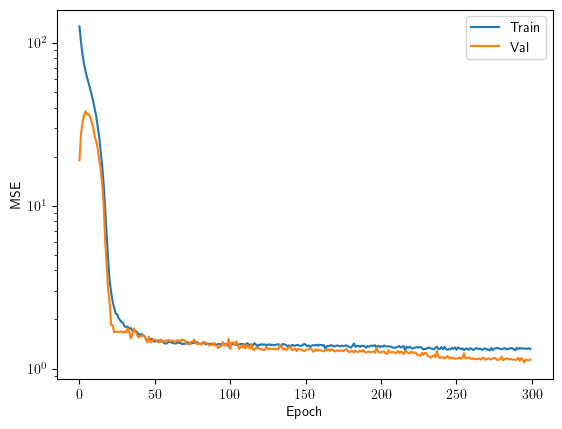}
        \caption{Res-Net}
        \label{train resnet}
    \end{subfigure}
  \caption{Training dynamics of the models tested. We can observe that DeepFDM has the smoothest training curve, leading to less variance in the results obtained.}
  \label{tranining}
\end{figure}

\section{Conclusion}
In this work, we introduced DeepFDM, a benchmark framework for comparing neural Partial Differential Equation (PDE) operators with traditional numerical solvers. While DeepFDM is not intended as a replacement for neural PDE solvers, it leverages the inherent structure of PDEs to offer improved accuracy, generalization, and interpretability, particularly in out-of-distribution (OOD) scenarios. Additionally, we proposed a method for generating and quantifying distribution shifts using the Hellinger distance, enabling robust performance evaluation across diverse PDE problems. Our results show that DeepFDM consistently outperforms neural operator methods in specific classes of PDEs, making it a valuable tool for benchmarking and advancing neural PDE operator research.

%\begin{acknowledgements}
%\end{acknowledgements}

\section*{Acknowledgements}
AO and GR were supported by Canada CIFAR AI Chairs and NSERC Discovery programs.

\bibliographystyle{alpha}
\bibliography{references.bib}

\newcommand{\etalchar}[1]{$^{#1}$}
\begin{thebibliography}{CORJO23}

\bibitem[CFL28]{courant1928partiellen}
Richard Courant, Kurt Friedrichs, and Hans Lewy.
\newblock {\"U}ber die partiellen differenzengleichungen der mathematischen
  physik.
\newblock {\em Mathematische annalen}, 100(1):32--74, 1928.

\bibitem[CFL67]{courant1967partial}
Richard Courant, Kurt Friedrichs, and Hans Lewy.
\newblock On the partial difference equations of mathematical physics.
\newblock {\em IBM journal of Research and Development}, 11(2):215--234, 1967.

\bibitem[{COM}23]{comsol2023}
{COMSOL}.
\newblock {\em COMSOL Multiphysics{\textregistered} v. 6.1}.
\newblock COMSOL AB, Stockholm, Sweden, 2023.
\newblock Accessed: 2024-09-26.

\bibitem[CORJO23]{cao2023residual}
L.~Cao, T.~O'Leary-Roseberry, P.K. Jha, and J.T. Oden.
\newblock Residual-based error correction for neural operator accelerated
  infinite-dimensional bayesian inverse problems.
\newblock {\em Journal of Computational Physics}, 2023.

\bibitem[Cra99]{cramer1999mathematical}
Harald Cram{\'e}r.
\newblock {\em Mathematical Methods of Statistics}, volume~26.
\newblock Princeton University Press, 1999.

\bibitem[CRBD18]{chen2018neural}
Ricky~TQ Chen, Yulia Rubanova, Jesse Bettencourt, and David~K Duvenaud.
\newblock Neural ordinary differential equations.
\newblock {\em Advances in neural information processing systems}, 31, 2018.

\bibitem[HHRS22]{huang2022efficient}
Daniel~Zhengyu Huang, Jiaoyang Huang, Sebastian Reich, and Andrew~M Stuart.
\newblock Efficient derivative-free bayesian inference for large-scale inverse
  problems.
\newblock {\em Inverse Problems}, 38(12):125006, 2022.

\bibitem[HR17]{Haber_2017}
Eldad Haber and Lars Ruthotto.
\newblock Stable architectures for deep neural networks.
\newblock {\em Inverse Problems}, 34(1):014004, dec 2017.

\bibitem[HZRS16]{he2016deep}
Kaiming He, Xiangyu Zhang, Shaoqing Ren, and Jian Sun.
\newblock Deep residual learning for image recognition.
\newblock In {\em Proceedings of the IEEE conference on computer vision and
  pattern recognition}, pages 770--778, 2016.

\bibitem[Isa06]{isakov2006inverse}
Victor Isakov.
\newblock {\em Inverse problems for partial differential equations}.
\newblock Springer, 2006.

\bibitem[JYHL24]{jiao2024solving}
A.~Jiao, Q.~Yan, J.~Harlim, and L.~Lu.
\newblock Solving forward and inverse pde problems on unknown manifolds via
  physics-informed neural operators.
\newblock {\em arXiv preprint arXiv:2407.05477}, 2024.

\bibitem[KKL{\etalchar{+}}21]{Karniadakis_2021}
George~Em Karniadakis, Ioannis~G. Kevrekidis, Lu~Lu, Paris Perdikaris, Sifan
  Wang, and Liu Yang.
\newblock Physics-informed machine learning.
\newblock {\em Nature Reviews Physics}, 3(6):422--440, may 2021.

\bibitem[LJK19]{lu2019deeponet}
Lu~Lu, Pengzhan Jin, and George~Em Karniadakis.
\newblock Deeponet: Learning nonlinear operators for identifying differential
  equations based on the universal approximation theorem of operators.
\newblock {\em arXiv preprint arXiv:1910.03193}, 2019.

\bibitem[LKA{\etalchar{+}}20a]{Li:2020aa}
Zongyi Li, Nikola Kovachki, Kamyar Azizzadenesheli, Burigede Liu, Kaushik
  Bhattacharya, Andrew Stuart, and Anima Anandkumar.
\newblock Fourier neural operator for parametric partial differential
  equations.
\newblock {\em arXiv preprint arXiv:2010.08895}, 10 2020.

\bibitem[LKA{\etalchar{+}}20b]{Li:2020ab}
Zongyi Li, Nikola Kovachki, Kamyar Azizzadenesheli, Burigede Liu, Kaushik
  Bhattacharya, Andrew Stuart, and Anima Anandkumar.
\newblock Neural operator: Graph kernel network for partial differential
  equations.
\newblock {\em arXiv preprint arXiv:2003.03485}, 03 2020.

\bibitem[LLMD18]{long2018pde}
Zichao Long, Yiping Lu, Xianzhong Ma, and Bin Dong.
\newblock Pde-net: Learning pdes from data.
\newblock In {\em International conference on machine learning}, pages
  3208--3216. PMLR, 2018.

\bibitem[LMMK21]{lu2021deepxde}
Lu~Lu, Xuhui Meng, Zhiping Mao, and George~Em Karniadakis.
\newblock Deepxde: A deep learning library for solving differential equations.
\newblock {\em SIAM review}, 63(1):208--228, 2021.

\bibitem[LMW{\etalchar{+}}12]{logg2012automated}
Anders Logg, Kent-Andre Mardal, Garth Wells, et~al.
\newblock Automated solution of differential equations by the finite element
  method: The fenics book.
\newblock {\em Lecture Notes in Computational Science and Engineering}, 84,
  2012.
\newblock Accessed: 2024-09-26.

\bibitem[LSZ{\etalchar{+}}22]{liu2022predicting}
Xin-Yang Liu, Hao Sun, Min Zhu, Lu~Lu, and Jian-Xun Wang.
\newblock Predicting parametric spatiotemporal dynamics by multi-resolution pde
  structure-preserved deep learning.
\newblock {\em arXiv preprint arXiv:2205.03990}, 2022.

\bibitem[LT09]{Stigg}
Stig Larsson and Vidar Thom{\'e}e.
\newblock {\em {Partial Differential Equations With Numerical Methods}},
  volume~45.
\newblock Springer, Chalmers University of Technology and University of
  Gothenburg 412 96 G{\"o}teborg Sweden, 2009.

\bibitem[Obe06]{oberman2006convergent}
Adam~M Oberman.
\newblock Convergent difference schemes for degenerate elliptic and parabolic
  equations: Hamilton--jacobi equations and free boundary problems.
\newblock {\em SIAM Journal on Numerical Analysis}, 44(2):879--895, 2006.

\bibitem[RBPK17]{rudy2017data}
Samuel~H Rudy, Steven~L Brunton, Joshua~L Proctor, and J~Nathan Kutz.
\newblock Data-driven discovery of partial differential equations.
\newblock {\em Science advances}, 3(4):e1602614, 2017.

\bibitem[RFB15]{ronneberger2015u}
Olaf Ronneberger, Philipp Fischer, and Thomas Brox.
\newblock U-net: Convolutional networks for biomedical image segmentation.
\newblock In {\em Medical image computing and computer-assisted
  intervention--MICCAI 2015: 18th international conference, Munich, Germany,
  October 5-9, 2015, proceedings, part III 18}, pages 234--241. Springer, 2015.

\bibitem[RH20]{Ruthotto2020}
Lars Ruthotto and Eldad Haber.
\newblock Deep neural networks motivated by partial differential equations.
\newblock {\em Journal of Mathematical Imaging and Vision}, 62(3):352--364, Apr
  2020.

\bibitem[RTH17]{doi:10.1137/16M1081063}
Lars Ruthotto, Eran Treister, and Eldad Haber.
\newblock jinv--a flexible julia package for pde parameter estimation.
\newblock {\em SIAM Journal on Scientific Computing}, 39(5):S702--S722, 2017.

\bibitem[SDK20]{shin2020convergence}
Yeonjong Shin, Jerome Darbon, and George~Em Karniadakis.
\newblock On the convergence and generalization of physics informed neural
  networks.
\newblock {\em arXiv e-prints}, pages arXiv--2004, 2020.

\bibitem[Stu10]{stuart2010inverse}
Andrew~M Stuart.
\newblock Inverse problems: a bayesian perspective.
\newblock {\em Acta numerica}, 19:451--559, 2010.

\bibitem[TD06]{taler2006solving}
Jan Taler and Piotr Duda.
\newblock {\em Solving direct and inverse heat conduction problems}.
\newblock Springer, 2006.

\bibitem[TPL{\etalchar{+}}22]{DBLP:conf/nips/TakamotoPLMAPN22}
Makoto Takamoto, Timothy Praditia, Raphael Leiteritz, Daniel MacKinlay,
  Francesco Alesiani, Dirk Pfl{\"{u}}ger, and Mathias Niepert.
\newblock Pdebench: An extensive benchmark for scientific machine learning.
\newblock In Sanmi Koyejo, S.~Mohamed, A.~Agarwal, Danielle Belgrave, K.~Cho,
  and A.~Oh, editors, {\em Advances in Neural Information Processing Systems
  35: Annual Conference on Neural Information Processing Systems 2022, NeurIPS
  2022, New Orleans, LA, USA, November 28 - December 9, 2022}, 2022.

\bibitem[TR19]{tzen2019neural}
Belinda Tzen and Maxim Raginsky.
\newblock Neural stochastic differential equations: Deep latent gaussian models
  in the diffusion limit.
\newblock {\em arXiv preprint arXiv:1905.09883}, 2019.

\bibitem[UBF{\etalchar{+}}20]{DBLP:conf/nips/UmBFHT20}
Kiwon Um, Robert Brand, Yun~(Raymond) Fei, Philipp Holl, and Nils Thuerey.
\newblock Solver-in-the-loop: Learning from differentiable physics to interact
  with iterative pde-solvers.
\newblock In Hugo Larochelle, Marc'Aurelio Ranzato, Raia Hadsell,
  Maria{-}Florina Balcan, and Hsuan{-}Tien Lin, editors, {\em Advances in
  Neural Information Processing Systems 33: Annual Conference on Neural
  Information Processing Systems 2020, NeurIPS 2020, December 6-12, 2020,
  virtual}, 2020.

\bibitem[VHG24]{DBLP:conf/iclr/VermaH024}
Yogesh Verma, Markus Heinonen, and Vikas Garg.
\newblock Climode: Climate and weather forecasting with physics-informed neural
  odes.
\newblock In {\em The Twelfth International Conference on Learning
  Representations, {ICLR} 2024, Vienna, Austria, May 7-11, 2024}.
  OpenReview.net, 2024.

\bibitem[Vir20]{2020SciPy-NMeth}
Pauli et~al Virtanen.
\newblock Scipy 1.0: Fundamental algorithms for scientific computing in python,
  Feb 2020.

\bibitem[Zha24]{zhang_neuralpdesolver_2024}
Chengyang Zhang.
\newblock Neural-pde-solver.
\newblock \url{https://github.com/bitzhangcy/Neural-PDE-Solver}, 2024.
\newblock Version 1.0.

\bibitem[ZLW22]{DBLP:conf/icml/ZhaoLW22}
Qingqing Zhao, David~B. Lindell, and Gordon Wetzstein.
\newblock Learning to solve pde-constrained inverse problems with graph
  networks.
\newblock In Kamalika Chaudhuri, Stefanie Jegelka, Le~Song, Csaba
  Szepesv{\'{a}}ri, Gang Niu, and Sivan Sabato, editors, {\em International
  Conference on Machine Learning, {ICML} 2022, 17-23 July 2022, Baltimore,
  Maryland, {USA}}, volume 162 of {\em Proceedings of Machine Learning
  Research}, pages 26895--26910. {PMLR}, 2022.

\bibitem[ZXL24]{zhang2024bilo}
R.Z. Zhang, X.~Xie, and J.~Lowengrub.
\newblock Bilo: Bilevel local operator learning for pde inverse problems.
\newblock {\em arXiv preprint arXiv:2404.17789}, 2024.

\end{thebibliography}

\end{document}